\documentclass[12pt,reqno,a4paper]{amsart}
\usepackage{amssymb}
\usepackage{amsmath}
\usepackage{amsthm}

\newcommand{\al}{\alpha}
\newcommand{\ga}{\gamma}

\newcommand{\eo}{{\varepsilon_0}}

\newcommand{\om}{\omega}

\newcommand{\be}{\beta}

\newcommand{\de}{{\delta}}

\DeclareMathOperator{\sgn}{sgn}

\newtheorem{theo}{Theorem}
\newtheorem{lem}{Lemma}

\begin{document}
\title[MSO limit laws for natural well orderings]{Monadic second order limit laws for  natural well orderings}
\author{Andreas Weiermann}
\address{Vakgroep Wiskunde: Analysis, logic and discrete mathematics\\
Krijgslaan 281 S8,\\
9000 Ghent,
Belgium} 
\date{}

\begin{abstract}
By combining classical results of B\"uchi, some elementary Tauberian theorems and some basic tools from
logic and combinatorics
we show that every
ordinal $\alpha$ with $\eo\geq \alpha\geq \omega^\omega$ 
satisfies a natural monadic second order limit law
and that every
ordinal $\alpha$ with $\omega^\omega>\alpha\geq \omega$ 
satisfies  a natural monadic second order Cesaro limit law.
In both cases we identify as usual $\al$ with the class of substructures $\{\be:\be<\al\}$.

We work in an additive setting where the norm function $N$ assigns to every ordinal $\alpha$ the number of occurrrences of the symbol $\omega$ in its Cantor normal 
form. This number is the same as the number of edges in the tree which is canonically associated with $\alpha$.

For a given $\al$ with $\om\leq \al\leq \eo$ the asymptotic probability of a monadic second order formula $\varphi$ from the language of linear orders is
$\lim_{n\to\infty} \frac{\#\{\be<\al: N\be=n\wedge \be\models \Phi\}}{\#\{\be<\al: N\be=n\}}$ if this limit exists. If this limit exists only in the Cesaro sense we speak
of the Cesaro asympotic probability of $\varphi$.

Moreover we prove monadic second order limit laws for the ordinal segments below below $\Gamma_0$ (where the norm function is extended appropriately)
and we indicate how this paper's results can be extended to larger ordinal segments and even to certain impredicative ordinal notation systems
having notations for uncountable ordinals. We also briefly indicate how to prove the corresponding multiplicative results for which the setting is defined relative to the Matula coding.

The results of this paper concerning ordinals not exceeding $\eo$ have been obtained partly in joint work with Alan R. Woods.

\end{abstract}
\keywords{Schur's Tauberian theorem, Hua's Tauberian theorem, B{\"u}chi's definability results, asymptotic density, ordinals, monadic second order logic, limit law, natural well orderings}
\subjclass{03F15, 05C30,60C05}

\maketitle
\section{Introduction} This paper concerns logical limit laws for infinite ordinals. It is based on methods and techniques from
the theory of logical limit laws for classes of finite structures and fundamental results about the monadic theory of ordinals by B\"uchi. 
For the analytic part we make use of techniques developed by Bell, Burris and Compton \cite{Bella,Bell, Burris}.

 In 2001 the author discussed the possibility of logical limit laws for ordinals with Kevin Compton at an AOFA-workshop in Tatihoo
and Compton very kindly suggested among other things to contact Alan Woods because Woods proved very general results about limit laws for finite trees \cite{Woods}. 
This initiated a very fruitful interaction between Woods and the author over the years.
Tragically Woods passed away untimely 
in december 2011. 

The cooperation with Woods led to a first publication about first order zero one laws and limit laws for ordinals \cite{WeiermannWoods} based 
on a mixture of results from \cite{Weiermann}, \cite{Woods} and \cite{Burris}.

In this article we move our focus from first order logic to monadic second order logic.
  B\"uchi already provided a very explicit description of the ordinal spectrum of monadic second order sentences. In this article
it is shown how this description can be combined with results from  \cite{Bella,Bell, Burris} and \cite{Weiermann}
to prove monadic second order limit laws and monadic second order Cesaro limit laws using elementary Tauberian methods.

Alan Woods had originally in mind to prove the monadic second order results regarding the ordinals not exceeding $\eo$ by using
Shelah's theory of additive colourings \cite{Shelah}. We believe that this will prove useful in future research which among other things 
could lead to an automata free proof of this paper's results.

\section{Some basic definitions} 
Let $\varepsilon_0$ be the least ordinal $\xi$ such that $\xi=\om^\xi$ where $\om$ refers to the first infinite ordinal.
The ordinal $\varepsilon_0$ plays an important role in the proof theory of first order arithmetic and related contexts (see, for example, \cite{Schuette})  but in this article
we consider this ordinal (which is identified with its segment of smaller elements) as an object which stands just on itself.

By seminal work of Cantor we know that every ordinal $\al<\eo$ can be written uniquely as
$\om^{\al_1}+\cdots+\om^{\al_n}$ where $\al_1\geq \ldots\geq \al_n$. This normal form allows to associate finite trees 
canonically to the ordinals below $\varepsilon_0$.

Indeed, by writing out the $\al_i$ hereditarily in a similar fashion every ordinal is associated with a unique
term representation. 
To each $\al<\eo$ we can therefore canonically associate a finite tree $T(\al)$ in a recursive manner as follows.
$T(0)$ is the singleton tree consisting only of its root and if $\al$ has Cantor normal form $\om^{\al_1}+\cdots+\om^{\al_n}$ (meaning that $\al_1\geq \ldots\geq \al_n$)
and if the $T(\al_i)$ are already constructed then let $T(\al)$ be the tree with immediate subtrees $T(\al_i)$ connected to a new root.
$T(\al)$ is then a rooted non planar finite tree.

Let $N(\al)$ be the number of edges in $T(\al)$ which is one plus the number of nodes in $T(\al)$. Then $N0=0$ and $N(\al)=n+N(\al_1)+\ldots+N(\al_n)$ if $\al$ has Cantor normal form
$\om^{\al_1}+\cdots+\om^{\al_n}$. In other words $N\al$ is the number of occurrences of $\om$ in the ordinal $\al$. 

The norm function $N$ is additive in the sense that $N(\al)=N(\om^{\al_1})+\ldots +N(\om^{\al_n})$.
For $\om\leq \be\leq \eo$ let $c_\be(n):=\#\{\al<\be:N(\al)=n\}.$ Then $c_\be(n)$ is a well defined natural number.
Morever using techniques from \cite{Bell, Burris} it has been shown in \cite{Weiermann} that $c_\be(n)\in RT_1$ if $\be<\eo$. 
The latter refers to terminology borrowed from \cite{Burris} and means
$\lim{n\to\infty}\frac{c_\be(n)}{c_\be(n+1)}=1$ so that the radius of convergence of the associated generating function is $1$.

To introduce limit laws for ordinals we work with relational (monadic second order) languages $L_<$ which come equipped with exactly one relation symbol for the less than
relation. We neither do allow constants nor function symbols. For any $L_<$ sentence $\varphi$ the semantics of $\al\models \varphi$ is defined in the natural way.

For an $L_<$ sentence $\varphi$ we can define $\de_\varphi^\be(n)=\frac{\#\{\al<\be: \al\models \varphi\wedge N(\al)=n\}}{c_\be(n)}$.
If $\de_\varphi^\be=\lim_{n\to\infty}\de_\varphi^\be(n)$ exists for all $\varphi$ under consideration then $\be$ fulfills a monadic second order limit law and when
$\lim_{n\to\infty}\de_\varphi^\be(n)$ exists in the Cesaro sense for all $\varphi$
under consideration  (which means that $\lim_{n\to\infty}\frac 1n \sum_{i=0}^n\de_\varphi^\be(i)$ exists) then $\be$ fulfills a monadic second order Cesaro limit law.
Clearly limit laws yield Cesaro limit laws but not vice versa.

For describing a corresponding multiplicative setting one might use the Matula coding $M:\eo\to\om$ which is defined as follows.
Let $M(0)=1$ and $M(\al)=p_{M(\al_1)}\cdots p_{M(\al_n)}$ if $\al$ has Cantor normal form $\om^{\al_1}+\cdots+\om^{\al_n}$.
Here $p_i$ refers to the $i$-th prime number starting with $p_0=2$.
This coding is multiplicative in the sense that $M(\al)=M(\om^{\al_1})\cdot \ldots \cdot M(\om^{\al_n})$.

For a sentence $\varphi$ we can define $\Delta_\varphi^\be(n)=\frac{\#\{\al<\be: \al\models \varphi\wedge M(\al)\leq n\}}{C_\be(n)}$.
Monadic second order Cesaro limit laws and monadic second order limit laws can be defined accordingly in the multiplicative setting.

\section{Cesaro limit densities for semi linear subsets of ordinals below $\om^\om$}
For this section let us fix an infinite $\be<\om^\om$. We first assume $\be=\om^{r+1}$.
We call a subset $L$ of $\be$ linear if there exists a double sequence $a_r,b_r,\ldots,a_0,b_0$ of non negative integers
such that $$L=\{\al=\om^r\cdot (a_r+b_r\cdot l_r)+\cdots+\om^0\cdot (a_0+b_0 \cdot l_0):(\forall i\leq r)[l_i<\om] \}.$$
We call a subset $L$ of $\be$ semi linear if it is a finite union of linear subsets of $\be$.

For a semi linear subset $L\subseteq \be$ let

$$D_L(n)=\frac{\#\{\al\in L:N(\al)=n\}}{\#\{\al<\be: N(\al)=n\}}.$$

We will show that $C-\lim_{\to\infty}D_L(n)$ exists in the Cesaro sense. 
This will be useful to show monadic second order limit laws for $\be$ in section 6.

In bypassing let us remark that in general the standard limit $\lim_{\to\infty}D_L(n)$ does not alway exist.
We can take $\be=\om$ and $L=\{\al=2\cdot l: l<\om\}$.
Then $D_L(2n)=1$ and $D_L(2n+1)=0$.

Let us first show that Cesaro limits exist for linear subsets of $\be$.

\begin{lem} Let $\be=\om^{r+1}$. Suppose that $L=\{\al=\om^r\cdot (b_r\cdot l_r)+\cdots+\om^0\cdot (b_0 \cdot l_0):(\forall i\leq r)[l_i<\om] \}$. If there exists an $i\leq r$ such that 
$b_i=0$ then $\lim_{n\to\infty}D_L(n)=0$.

\end{lem}
\begin{proof}
 We can regard $\be$, hence the set of ordinals less than $\be$, as an additive number system in the sense of \cite{Burris} with set of primes given by
$\{\om^j:0\leq j\leq r\}$. The norm function for this number system is given by $N$ and the addition function is provided by the natural sum of ordinals.

Let $L'=\{\al=\om^r\cdot l_r+\cdots +\om^{i+1} \cdot l_{i+1} +\om^i\cdot 0+\om^{i-1} \cdot l_{i-1}+\cdots +\om^0\cdot l_0:(\forall i\leq r)[l_i<\om]\}$.
Then $L'$ can be considered as a partition set with small exponent $0$ for the partition element $\{\om^i\}$.

Since $c_\be(n)\in RT_1$ we conclude that $\lim_{n\to\infty}\frac {\#\{\al\in L':N\al=n\}}{c_\be(n)}=0$ by Compton's theorem 4.2 in \cite{Burris}.

Since $L\subseteq L'$ we see that $\lim_{n\to\infty}D_L(n)=0$.

\end{proof}

So we are left with the case that all the $b_i$ are different from zero.

Let us first recall Hua's theorem (see, for example, theorem 2.48 in \cite{Burris}).

\begin{theo}
Suppose the additive number system ${\mathcal A}$ has finite rank $r$ with
$d$ the gcd of $supp(p(n))$. Let $supp( p(n))=\{d_1,\ldots,d_k\}$ and let
$p(d_i)=m_i$.
Then
$a(nd)\sim \frac{d^r}{(r-1)!\prod d_i^{m_i}}\cdot n^{r-1}$ as $n\to\infty$.
\end{theo}

\begin{lem} 

Let $\be=\om^{r+1} $ and 
$L=\{\al=\om^r\cdot (b_r\cdot l_r)+\cdots+\om^0\cdot (b_0 \cdot l_0):(\forall i\leq r )[l_i<\om] \}$
where no $b_i$ is zero. Let 

$d:=gcd((r+1)\cdot b_r,\ldots,1\cdot b_0)$.
Then

$$\lim_{n\to\infty}D_L(d\cdot n)=\frac{d}{\prod_{i=1}^rb_i}.$$
\end{lem} 
\begin{proof}
Since all $b_i$ are non zero we can regard $L$ as an additive number system with primes $\om^i\cdot b_i\; (0\leq i\leq r)$. Then $L$ has rank $r+1$.
Moreover, $d=\gcd(supp (p(n))$ where $supp (p(n))=\{N(\om^r\cdot b_r),N(\om^{r-1})\cdot b_{r-1},\ldots, N(\om^0\cdot b_0)\}$.
Then $d=\gcd((r+1)\cdot b_r,\ldots,1)$.
Hence Hua's theorem yields
$\#\{\al\in L: N\al=n\cdot d\}\sim  \frac{d^{r+1}}{r!\cdot \prod_{i=0}^r (i+1)\cdot b_i}\cdot n^r$.

Also $\be$ itself can be seen as an additive number system with primes $\om^i\; (0\leq i\leq r)$. Then $\be$ has rank $r+1$.
Since $\gcd(N\om^r,\ldots,N\om^0)=1$ a second application of Hua's result yields 
$c_\be(n)\sim \frac{1}{r!\cdot \prod_{i=0}^r (i+1)}\cdot n^r$.

Putting things together we find $\lim_{n\to\infty}D_L(d\cdot n)=\lim_{n\to\infty}\frac{\#\{\al\in L:N\al=nd\}}{c_\be(nd)}=\frac{d}{\prod_{i=1}^rb_i}$.

\end{proof}

\begin{lem}
Let $\be=\om^{r+1} $ and 
$L=\{\al=\om^r\cdot (b_r\cdot l_r)+\cdots+\om^0\cdot (b_0 \cdot l_0):(\forall i\leq r)[l_i<\om ]\}$
where no $b_i$ is zero. Let 

$d:=gcd((r+1)\cdot b_r,\ldots,1\cdot b_0)$.
Then

$C-\lim_{n\to\infty}D_L(d\cdot n)=\frac{1}{\prod_{i=1}^rb_i}$.

\end{lem}
\begin{proof}

Let $0<e<d$. Then there will be no $\al\in L$ such that $N\al=nd+e$.
Otherwise such an $\al$ has the form $\al=\om^r\cdot (b_r\cdot l_r)+\cdots+\om^0\cdot (b_0 \cdot l_0)$.
From $N\al=nd+e$ we would conclude $N(\om^r\cdot b_r)\cdot l_r+\cdots+N(\om^0\cdot b_0)\cdot l_0-nd=e.$
Then by theorem 4.1 in \cite{vanzurGathen} the number $\gcd(N(\om^r\cdot b_r),\ldots,N(\om^0\cdot b_0),d)$ would divide $e$ which is absurd.
Therefore $D_L(n\cdot d+e)=0$. 
Combining this with $\lim_{n\to\infty}D_L(d\cdot n)=\frac{d}{\prod_{i=1}^rb_i}$ we find
that $C-\lim_{n\to\infty}D_L(n)=\lim_{n\to\infty}\frac 1n \sum_{i=1}^n D_L(i)=\frac{1}{\prod_{i=1}^rb_i}$.
\end{proof}

We now consider semi linear sets where the $a_i$ might be non zero.

\begin{lem}

Let $\be=\om^{r+1} $ and let
$L=\{\al=\om^r\cdot (a_r+b_r\cdot l_r)+\cdots+\om^0\cdot (a_0+b_0 \cdot l_0):(\forall i\leq r)[l_i<\om]\}$ be a semi linear subset of $\be$.
\begin{enumerate}
\item If some $b_i=0$ then $\lim_{n\to\infty}D_L(n)=0$.
\item If all $b_i$ are non zero then $C-\lim_{n\to\infty}D_L(d\cdot n)=\frac{1}{\prod_{i=1}^rb_i}$.
\end{enumerate}

\end{lem}

\begin{proof}

Let $L'=\{\al=\om^r\cdot (b_r\cdot l_r)+\cdots+\om^0\cdot (b_0 \cdot l_0):(\forall i\leq r)[l_i<\om]\}$.

Then $\al=\om^r\cdot (a_r+b_r\cdot l_r)+\cdots+\om^0\cdot (a_0+b_0 \cdot l_0)\in L$ has norm $n$ if
$\al'=\om^r\cdot (b_r\cdot l_r)+\cdots+\om^0\cdot (b_0 \cdot l_0)\in L'$ has norm $n-N(\om^r\cdot a_r)-\ldots-N(\om^0\cdot a_0)$.
Hence $D_L(n+(r+1)\cdot a_r+\cdots +a_0)=D_{L'}(n)$.
Therefore the previous results obviously carry over from $L'$ to $L$ since $c_\be(n)\in RT_1$.

\end{proof}

Let us now consider semi linear subsets.

\begin{lem}  Let $\be=\om^{r+1}$. If $L$ and $L'$ are linear subsets of $\be$ then $L\cap L'$ is either empty or again a  linear set.
The same conclusion holds for any finite intersection of linear sets.
\end{lem}
\begin{proof}
The second assertion follows from the first by an obvious induction.

Now assume that 
$L=\{\al=\om^r\cdot (a_r+b_r\cdot l_r)+\cdots+\om^0\cdot (a_0+b_0 \cdot l_0):(\forall i\leq r)[l_i<\om] \}$
and
$L'=\{\al=\om^r\cdot (a'_r+b'_r\cdot l_r)+\cdots+\om^0\cdot (a'_0+b'_0 \cdot l_0):(\forall i\leq r)[l_i<\om] \}$.

Then $\al\in L\cap L'$ iff
$\al=\om^r\cdot (a_r+b_r\cdot l_r)+\cdots+\om^0\cdot (a_0+b_0 \cdot l_0)=
\om^r\cdot (a'_r+b'_r\cdot l'_r)+\cdots+\om^0\cdot (a'_0+b'_0 \cdot l'_0)$
iff 
for all $j\leq r$ we have $a_j+b_jl_j=a'_j+b'_jl'_j$
iff 
for all $j\leq r$ we have $a_j-a'_j=b'_jl'_j-b_jl_j$.

So if there would exist an $j\leq r$ such that $\gcd(b_j,b'_j)$ does not divide $a_j-a'_j$
then by theorem 4.1 of \cite{vanzurGathen} we would obtain $L\cap L'=\emptyset$.
So let us assume that  $\gcd(b_j,b'_j)$ does divide $a_j-a'_j$ for all $j\leq r$.

For the moment let us fix a $j\leq r$. 
Assume that $(\star)\quad a_j+b_jl_j=a'_j+b'_jl'_j$. Let $h_j:=\gcd(b_j,b'_j)$.
Assume $l_j^*$ is the minimal non negative left hand side of a non negative solution $l_j,l'_j$ of $(\star)$.
Let $l_j'^*$ be the right hand side of such a solution. 
By theorem 4.1 in \cite{vanzurGathen} the set of all integer solution of $(\star)$ has the form
$(l^*_j,l'^*_j)+{\mathbb Z}(\frac {b_j'}{h_j}, \frac {b_j}{h_j})$. But since we are only interested in non negative solutions and since $(l^*_j,l'^*_j)$ is left minimal
all non negative solutions of $(\star)$ have the form $(l^*_j,l'^*_j)+{\mathbb N}(\frac {b_j'}{h_j}, \frac {b_j}{h_j})$.
A typical right hand side of $(\star)$ thus has the form
$a_j+b_j(l_j^*+l_j\frac {b_j'}{h_j})=a_j+b_jl_j^*+l_jb_j \frac {b_j'}{h_j}$.
This analysis can be done for the other $j$'s as well and therefore
$L\cap L'=\{\al=  \om^r\cdot (a_r+b_rl_r^*+b_r\frac {b_r'}{h_r} l_r)+\cdots+ \om^0\cdot (a_0+b_0l_0^*+b_0 \frac {b_0'}{h_0} l_0):(\forall i\leq r)[l_i<\om]\}$
which is a linear set.
\end{proof}

\begin{lem} Let $\be=\om^{r+1}$. Then Cesaro limit densities exist for all semi linear subsets of $\be.$
\end{lem}
\begin{proof}
Cesaro limits distribute over finite sums. 
The counting function for a given semi linear set can be calculated as a finite sum of counting functions of linear sets (a sum with possibly negative integer coefficients).
This follows from the inclusion exclusion principle stating that for finite sets $A_i$ with $i$ ranging over a finite set $I$ we have that
$$\lvert \bigcup_{i\in I} A_i\rvert=\sum_{k=1}^n (-1)^{k-1} \sum_{I\in {\{1,\ldots,n\}}\choose {k}}\lvert \bigcap_{i\in I} A_i\rvert.$$
\end{proof}

Finally let us consider a general $\be$ of the form $\be=\om^r\cdot c_r+\cdots +\om^0\cdot c_0$.
Then the set of ordinals $\al$ less than $\be$ can be written as a disjoint union
over sets $L_{j,k_j}:=\{\al=\om^r\cdot c_r+\cdots+ \om^j\cdot (c_j-k_j)+\om^{j-1}l_{j-1}+\cdots +\om^0\cdot  l_0:(\forall i\leq j)[l_i<\om] \}$ where $j\leq r$ and $k_j<c_j$

Then $\al=\om^r\cdot c_r+\cdots+ \om^j\cdot (c_j-k_j)+\om^{j-1}l_{j-1}+\om^0\cdot  l_0\in L_{j,k_j}\iff \om^{j-1}l_{j-1}+\cdots+\om^0\cdot  l_0<\om^r$.

If $L=\{\al=\om^r\cdot (a_r+b_r\cdot l_r)+\cdots+\om^0\cdot (a_0+b_0 \cdot l_0):(\forall i\leq r)[l_i<\om]  \}$ is a linear subset of $\be$ then
$b_r=0$ since $c_r<\om$. Moreover $a_r=c_r-k_r$ for some $k_r$. 
Let $L'=\{\al=\om^{r-1}\cdot (a_{r-1}+b_{r-1}\cdot l_{r-1})+\cdots+\om^0\cdot (a_0+b_0 \cdot l_0):(\forall i\leq r-1)[l_i<\om]  \}$.
Then $L'$ is a linear subset of $\om^r$.

Hence $\#\{\al\in L: N\al=N(\om^r\cdot a_r)+n\}=\#\{\al\in L': N\al=n\}$.
Since $c_\be(n)\in RT_1$ we conclude that the Cesaro density for $L$ exists
since the Cesaro limit for $L'$ exists.
Since by the same proof as before also in this situation the intersection of linear sets is either empty or again a linear set we have proved.

\begin{theo}
Let $\om\leq \be<\om^\om$. Then Cesaro limit densities exist for all semi linear subsets of $\be.$
\end{theo}
\begin{theo}
Let $\om\leq \be<\om^\om$. Then the Cesaro limit densities are rational  for all semi linear subsets of $\be.$
\end{theo}
\begin{proof}
The limit densities resulting in our setting from Hua's theorem for linear sets are rational numbers.
The resulting limiting densities for semilinear sets are formed by taking finite sums with integer coefficients.
This yields the assertion.

\end{proof}

\section{Limit laws for semi linear sets of ordinals stretching above $\om^\om$}
For this section let us fix an ordinal $\be$ with $\eo>\be\geq\om^\om$.
Let us first concentrate on the case where $\be$ of the form $\be=\om^{\ga}$ where $\ga\geq \om$.
Then $\be$ is closed under ordinal addition and it forms with the standard norm function
an additive number system with respect to the natural sum operation.
This additive number system has primes given by $\{\om^\al:\al<\ga\}$
and it has an infinite rank.

As before, we call a subset $L$ of $\be$ linear if there exists non negative integer $r$, the length of $L$, and a double sequence $a_r,b_r,\ldots,a_0,b_0$ of non negative integers
such that $L=\{\al=\om^{r+1}\cdot \al_0+\om^r\cdot (a_r+b_r\cdot l_r)+\cdots+\om^0\cdot (a_0+b_0 \cdot l_0):(\forall i\leq r)[l_i<\om]\wedge \al_0<\be \}$.
We call a subset $L$ of $\be$ semi linear if it is a finite union of linear subsets of $\be$.

Obviously, if a linear set is defined with respect to length $r$ and a double sequence $a_r,b_r,\ldots,a_0,b_0$ then
for any $s\geq r$ we can put $a_l=0$ and $b_l=1$ for $r<l\leq s$ to obtain a representation using the sequence $a_s,b_s,\ldots,a_0,b_0$.
So in forming finite unions and intersections of linear sets we can always assume that the linear sets have the same lengths.

For a semi linear subset $L\subseteq \be$ let

$$D_L(n)=\frac{\#\{\al\in L:N(\al)=n\}}{\#\{\al<\be: N(\al)=n\}}.$$

We will show that $\lim_{\to\infty}D_L(n)$ always exists in the usual sense. 
This will be useful to show monadic second order limit laws for the segment of ordinals determined by $\be$ in section 6.

\begin{lem} Let $\be=\om^{\ga}$ where $\eo>\ga\geq \om$.
Suppose that $L=\{\al=\om^{r+1}\cdot \al_0+\om^r\cdot (b_r\cdot l_r)+\cdots+\om^0\cdot (b_0 \cdot l_0):(\forall i\leq r)[l_i<\om]\wedge \al_0<\be \}$. If there exists an $i\leq r$ such that 
$b_i=0$ then $\lim_{n\to\infty}D_L(n)=0$.
\end{lem}
\begin{proof}
As before we can regard $\be$, hence the set of ordinals less than $\be$, as an additive number system in the sense of \cite{Burris} with set of primes given by
$\{\om^j:0\leq j<\ga\}$. The norm function for this number system is given by $N$ and the addition function is provided by the natural sum.

Let $L'=\{\al=\om^{r+1}\cdot \al_0+\om^r\cdot l_r+\cdots +\om^{i+1} \cdot l_{i+1} +\om^i\cdot 0+\om^{i-1} \cdot l_{i-1}+\cdots +\om^0\cdot l_0:\al_0<\be\wedge (\forall i\leq r)[l_i<\om]\}$.
Then $L'$ can be considered as a partition set with small exponent $0$ for the partition element $\{\om^i\}$.

Since $c_\be(n)\in RT_1$ we conclude that $\lim_{n\to\infty} \frac{\#\{\al\in L':N\al=n\}}{c_\be(n)}=0$ by Compton's theorem 4.2 in \cite{Burris}.

Since $L\subseteq L'$ we see that $\lim_{n\to\infty}D_L(n)=0$.
\end{proof}

So we are left with the case that all the $b_i$ are different from zero.

Let us now recall Schur's theorem (theorem 3.42 in \cite{Burris}).

\begin{theo}
Let $S(x),T(x)$ be two power series such that for some $\rho\geq 0$
\begin{enumerate}
\item $T(x)\in RT_\rho$,  and
\item $S(x)$ has radius of convergence $\rho_s$ greater than $\rho$.
\end{enumerate}

Then $\lim{[x^n](S(x)\cdot T(x)}{[x^n]T(x)}=S(\rho)$.

\end{theo}
Here $[x^n]T(x)$ refers to the $n$-th coefficient of the power series $T(x)$ and 
$[x^n]S(x)$ is defined correspondingly. $T(x)\in RT_\rho$ means that 
$\lim_{n\to\infty}{[x^n]T(x)}{[x^{n+1}]T(x)}=\rho$.

\begin{lem}
Let $\om^\om\leq \be=\om^\ga<\eo$ and ${L}_\be:=\{\al:\al=\om^{r+1}\cdot\al_0+\om^r\cdot (b_r\cdot l_r)+\cdots+\om^0\cdot (b_0\cdot l_0):\al_0<\be\wedge (\forall i\leq r)[l_i<\om]\}$
where all $b_i$ are different from zero. Let $l_\be(n)=\#\{\al\in L_\be:N\al=n\}$.
Then
$\lim_{n\to\infty} D_L(n)=\lim_{n\to\infty}\frac{l_\be(n)}{c_\be(n)}=\frac{1}{S(1)}=\frac{1}{b_r\cdot \ldots\cdot b_0}$.
\end{lem}

\begin{proof}

Let $l_\be(n):=\#\{\al\in L_\be:N\al=n\}$.
The set ${L}_\be$
can be seen as an additive number system with primes in the set $\tilde{P}:=\{\om^\xi:\ga>\xi>r\}\cup\{\om^i\cdot b_i:i\leq r\}$.

Let $S(x)={(1+\cdots+x^{(r+1)(b_r-1)})\cdot \ldots\cdot  (1+\cdots+x^{1\cdot (b_0-1)})}$ and $T(x)=\sum l_\be(n)x^n$.
By theorem 2.20 in \cite{Burris} we obtain for real numbers $x<1$
\begin{eqnarray*}
&&S(x)\cdot T(x)\\
&=&S(x)\cdot \sum l_\be(n)x^n\\
&=&S(x)\prod_{p\in \tilde{P}}  (1-x^{N(p)})^{-1}\\
&=&(\prod_{\om^\xi:\ga>\xi>r} (1-x^{N(\om^ \xi)})^{-1})\cdot  S(x)\cdot (1-x^{N(\om^r\cdot b_r)})^{-1}\cdot \ldots\cdot
(1-x^{N(\om^0\cdot b_0)})^{-1}\\
&=&\prod_{\om^\xi:\ga>\xi>r}  (1-x^{N(\xi)+1})^{-1}\cdot (1-x^{r+1})^{-1}\cdot \ldots\cdot
(1-x^{1})^{-1} \\
&=&\sum c_\be(n)x^n.
\end{eqnarray*}

The radius of convergence of $S(x)$ is infinite, hence bigger than $1$ 
and so Schur's Tauberian Theorem is applicable and yields
$\lim_{n\to\infty}\frac {c_\be(n)} {l_\be(n)}=S(1)={b_r\cdot \ldots\cdot b_0}$.
By taking inverses this yields
$\lim_{n\to\infty} D_L(n)=\lim_{n\to\infty}\frac{l_\be(n)}{c_\be(n)}=
\frac{1}{b_r\cdot \ldots\cdot b_0}$.
\end{proof}

We now consider semi linear sets where the $a_i$ might be non zero.

\begin{lem}

Let $\om^\om\leq\be=\om^{\ga} <\eo$ and let
$L=\{\al=\om^{r+1}\cdot \al_0+\om^r\cdot (a_r+b_r\cdot l_r)+\cdots+\om^0\cdot (a_0+b_0 \cdot l_0):\al_0<\be\wedge (\forall i\leq r)[l_i<\om] \}$ be a semi linear subset of $\be$.
\begin{enumerate}
\item If some $b_i=0$ then $\lim_{n\to\infty}D_L(n)=0$.
\item If all $b_i$ are non zero then $C-\lim_{n\to\infty}D_L(d\cdot n)=\frac{1}{\prod_{i=1}^rb_i}$.
\end{enumerate}

\end{lem}
\begin{proof}
Let $L'=\{\al=\om^{r+1}\cdot \al_0+\om^r\cdot (b_r\cdot l_r)+\cdots+\om^0\cdot (b_0 \cdot l_0):\al_0<\be\wedge (\forall i\leq r)[l_i<\om]\}$.

Then $\al=\om^{r+1}\cdot \al_0+\om^r\cdot (a_r+b_r\cdot l_r)+\cdots+\om^0\cdot (a_0+b_0 \cdot l_0)\in L$ has norm $n$ if
$\al'=\om^{r+1}\cdot \al_0+\om^r\cdot (b_r\cdot l_r)+\cdots+\om^0\cdot (b_0 \cdot l_0)\in L'$ has norm $n-N(\om^r\cdot a_r)-\ldots-N(\om^0\cdot a_0)$.
Hence $D_L(n+(r+1)\cdot a_r+\cdots +a_0)=D_{L'}(n)$.

Since ${c_\be(n)}\in RT_1$ we see 
$c_\be(n+(r+1)\cdot a_r+\cdots +a_0)\sim c_\be(n)$ as $n\to\infty$.

Therefore the previous results easily carry over from $L'$ to $L$.
\end{proof}

Let us now consider semi linear subsets of $\be=\om^{\ga}$.

\begin{lem} Let $\om^\om\leq \be=\om^\ga<\eo$. If $L$ and $L'$ are linear subsets of $\be$ then $L\cap L'$ is either empty or again a  linear set.
The same conclusion holds for any finite intersection of linear sets.
\end{lem}
\begin{proof}

The proof from the last section carries over immediately.

\end{proof}
\begin{lem} Let $\om^\om\leq \be=\om^\ga<\eo$. Then limit densities exist for all semi linear subsets of $\be.$
\end{lem}
\begin{proof} Standard limits distribute over finite sums. 
The counting function for a given semi linear set can be calculated as a finite sum counting functions of linear sets (a sum with possibly negative integer coefficients).
This follows from the inclusion exclusion principle.
\end{proof}

Finally let us consider a general $\be$ of the form $\be=\om^{\ga_1}\cdot d_1+\cdots +\om^{\ga_s}d_s$.
Then the set of ordinals $\al$ less than $\be$ can be written as a disjoint union
over sets $L_{j,k_j}:=\{\al=\om^{\ga_1}\cdot d_1+\cdots +\om^{\ga_j}(d_j-k_j)+\de$ where $\de<\om^{\ga_j}$.

Then $\al=\om^{\ga_1}\cdot d_1+\cdots +\om^{\ga_j}(d_j-k_j)+\de\in L_{j,k_j}\iff \de<\om^{\ga_j}$.

Since $c_\be(n)\in RT_1$ we conclude that the density for $L$ exists for all linear sets.

Since by the same proof as before also in this situation the intersection of linear sets is either empty or again a linear set we have proved.

\begin{theo}
Let $\om^\om\leq \be<\eo$. Then  limit densities exist for all semi linear subsets of $\be.$
\end{theo}

\begin{theo} Let $\om^\om\leq \be<\eo$. Then the limit densities are rational numbers for all semi linear subsets of $\be.$
\end{theo}
\begin{proof}
 The limit densities resulting in our setting from Schur's theorem for linear sets are rational numbers.
The resulting limiting densities for semi linear sets are formed by taking finite sums with integer coefficients.
This yields the assertion.
\end{proof}

\section{Limit laws for ordinals at least a big as $\eo$}

For this section let us fix $\eo\leq \be=\om^\ga$. A typical choice would be $\be=\eo$ or $\be=\Gamma_0$ (see, for example, \cite{Schuette} for a definition) or the Bachmann Howard ordinal (see, for example, \cite{Buchholz} for a definition).
We assume that $c_\be(n)\in RT_\rho$ for $\rho<1$. This is our standing assumption and will be true for all notations systems
known from the literature. For $\eo$ this follows because enumerating ordinals below $\eo$ comes down to counting finite rooted non planar trees.
The generating function for counting these trees has radius of convergence smaller than one by \cite{Otter}.

Counting ordinals below $\Gamma_0$ comes down to counting $2$-trees. 
Here the norm function satisfies
$N(\varphi\al\be+\ga)=1+N\al+N\be+N\ga$ where $\varphi\al\be+\ga$ is in Cantor normal form and $\varphi$ denotes the binary fixed point free Veblen function.
The resulting generating function for counting these trees has radius of convergence smaller than one by \cite{Leroux,Bondt}.

In general the generating function for counting such ordinals has radius of convergence smaller than one by the general theory for counting trees as for example explained in \cite{Bellb,Flajolet}.

The ordinals below $\be$ form again an additive number system with primes given by $\{\om^\al:\al<\ga\}$.
This system has an infinite rank.

We can define linear and semilinear sets as in the last section and the analysis proceeds exactly as before.
The only difference is that we are now no longer in the $RT_1$ case.

So we call a subset $L$ of $\be$ linear if there exists a non negative integer $r$, the length of $L$, and a double sequence $a_r,b_r,\ldots,a_0,b_0$ of non negative integers
such that $L=\{\al=\om^{r+1}\cdot \al_0+\om^r\cdot (a_r+b_r\cdot l_r)+\cdots+\om^0\cdot (a_0+b_0 \cdot l_0):(\forall i \leq r)[l_i<\om]\wedge \al_0<\be \}$.
We call a subset $L$ of $\be$ semi linear if it is a finite union of linear subsets of $\be$.

For a semi linear subset $L\subseteq \be$ let

$$D_L(n)=\frac{\#\{\al\in L:N(\al)=n\}}{\#\{\al<\be: N(\al)=n\}}.$$

We will show that $\lim_{\to\infty}D_L(n)$ exists in the usual sense. 
This will again be useful to show monadic second order limit laws for the segment of ordinals determined by $\be$.
In contrast to the previous section let us remark that
$L=\{\al=\om^{r+1}\cdot \al_0+\om^r\cdot (a_r+b_r\cdot l_r)+\cdots+\om^0\cdot (a_0+b_0 \cdot l_0):(\forall i\leq r)[l_i<\om]\wedge \al_0<\be \}$ and there exists an $i\leq r$ such that 
$b_i=0$ does not imply $\lim_{n\to\infty}D_L(n)=0$. The reason is that we cannot apply Compton's theorem 4.2 in \cite{Burris}.

In the new setting we still can apply Schur's theorem and it will also cover the case where some $b_i=0$.

\begin{lem}
Let $\eo\leq\be=\om^\ga$ and assume that $\be$ forms an additive number system with respect to a natural norm function which extends the norm function for the ordinals less than $\eo$
and where the primes are given by the set $P:=\{\om^\al:\al<\ga\}$.
Assume that the associated generating function has radius of convergence $\rho$ strictly smaller than one.
Let $${L}_\be:=\{\al:\al=\om^{r+1}\cdot\al_0+\om^r\cdot (b_r\cdot l_r)+\cdots+\om^0\cdot (b_0\cdot l_0):(\forall i \leq r)[l_i<\om]\wedge \al_0<\be\}$$
and let $l_\be(n)=\#\{\al\in L_\be:N\al=n\}$. Moreover
let $$S(x):=\prod_{i:b_i>0}(1+\cdots+x^{(i+1)(b_i-1)})\cdot \prod_{i:b_i=0}\frac1{  1-x^{i}}.$$ Then
$$\lim_{n\to\infty} D_L(n)=\lim{n\to\infty}\frac{l_\be(n)}{c_\be(n)}=\frac{1}{S(\rho)}.$$
\end{lem}

\begin{proof}
The set ${L}_\be:=\{\al:\al=\om^{r+1}\cdot\al_0+\om^r\cdot (b_r\cdot l_r)+\cdots+\om^0\cdot (b_0\cdot l_0):(\forall i \leq r)[l_i<\om]\wedge \al_0<\be\}$
can be seen as an additive number system with primes in the set $\tilde{P}:=\{\om^\xi:\ga>\xi>r\}\cup\{\om^i\cdot b_i:i\leq r\}$.

Let $T(x):=\sum_{n=0}^\infty  \#\{\al\in L_\be:N\al=n\} x^n$.

By theorem 2.20 in \cite{Burris} we obtain
\begin{eqnarray*}
&&S(x)\cdot T(x)\\
&=&S(x)\prod_{p\in \tilde{P}}  (1-x^{N(p)})^{-1}\\
&=&(\prod_{\om^\xi:\ga>\xi>r}  (1-x^{N(\om^ \xi)})^{-1})\cdot S(x)\cdot (1-x^{N(\om^r\cdot b_r)})^{-1}\cdot \ldots\cdot
(1-x^{N(\om^0\cdot b_0)})^{-1}\\
&=&\prod_{\om^\xi:\ga>\xi>r}  (1-x^{N(\xi)+1})^{-1}\cdot (1-x^{r+1})^{-1}\cdot \ldots\cdot
(1-x^{1})^{-1}\cdot \\
&=&\sum c_\be(n)x^n.
\end{eqnarray*}

The radius of convergence of $S(x)$ is at least as big as $1$ 
and so Schur's Tauberian Theorem is applicable and yields
$\lim_{n\to\infty}\frac{c_\be(n)} {l_\be(n)}=S(\rho)$.
By taking inverses this yields
$\lim_{n\to\infty} D_L(n)=\lim_{n\to\infty}\frac{l_\be(n)}{c_\be(n)}=\frac{1}{S(\rho)}$.
(Note that $S(\rho)$ is defined since the radius of convergence of $S$ is strictly bigger than one.)
\end{proof}

We now consider semi linear sets where the $a_i$ might be non zero.

\begin{lem}
Let $\eo\leq\be=\om^\ga$ and assume that $\be$ forms an additive number system with respect to a natural norm function which extends the norm function for the ordinals less than $\eo$
and where the primes are given by the set $P:=\{\om^\al:\al<\ga\}$.
Assume that the associated generating function has radius of convergence $\rho$ strictly smaller than one. Let
$L=\{\al=\om^{r+1}\cdot \al_0+\om^r\cdot (a_r+b_r\cdot l_r)+\cdots+\om^0\cdot (a_0+b_0 \cdot l_0):(\forall i\leq r)[l_i<\om]\wedge \al_0<\be \}$ be a semi linear subset of $\be$.
Then the limiting density for $L$ exists.
\end{lem}
\begin{proof}
Let $L'=\{\al=\om^{r+1}\cdot \al_0+\om^r\cdot (b_r\cdot l_r)+\cdots+\om^0\cdot (b_0 \cdot l_0):(\forall i\leq r)[l_i<\om]\wedge \al_0<\be \}$.
Then $\al=\om^{r+1}\cdot \al_0+\om^r\cdot (a_r+b_r\cdot l_r)+\cdots+\om^0\cdot (a_0+b_0 \cdot l_0)\in L$ has norm $n$ if
$\al'=\om^{r+1}\cdot \al_0+\om^r\cdot (b_r\cdot l_r)+\cdots+\om^0\cdot (b_0 \cdot l_0)\in L'$ has norm $n-N(\om^r\cdot a_r)-\ldots-N(\om^0\cdot a_0)$.
Hence $D_L(n+(r+1)\cdot a_r+\cdots +a_0)=D_{L'}(n)$.

Since ${c_\be(n)}\in RT_\rho $ we see 
$c_\be(n+(r+1)\cdot a_r+\cdots +a_0)\sim \rho^{(r+1)\cdot a_r+\cdots +a_0}\cdot c_\be(n)$ as $n\to\infty$.

Therefore the previous results easily carry over from $L'$ to $L$.

Let us now consider semi linear subsets of $\be=\om^{\ga}$.

\begin{lem} Let $\om^\om\leq \be=\om^\ga<\eo$. If $L$ and $L'$ are linear subsets of $\be$ then $L\cap L'$ is either empty or again a  linear set.
The same conclusion holds for any finite intersection of linear sets.
\end{lem}
Proof. The proof from the last section carries over immediately.\hfill $\Box$

\begin{lem} Let $\om^\om\leq \be=\om^\ga<\eo$.  Then limit densities exist for all semi linear subsets of $\be.$
\end{lem}
Proof. Standard limits distribute over finite sums. 
The counting function for a given semi linear set can be calculated as a finite sum counting functions of linear sets (a sum with possibly negative integer coefficients).
This follows again from the inclusion exclusion principle.
\end{proof}

Finally let us consider a general $\be$ of the form $\be=\om^{\ga_1}\cdot d_1+\cdots \om^{\ga_s}d_s$.
Then the set of ordinals $\al$ less than $\be$ can be written as a disjoint union
over sets $L_{j,k_j}:=\{\al=\om^{\ga_1}\cdot d_1+\cdots \om^{\ga_j}(d_j-k_j)+\de:\de<\om^{\ga_j}\}$.

Then $\al=\om^{\ga_1}\cdot d_1+\cdots \om^{\ga_j}(d_j-k_j)+\de\in L_{j,k_j}\iff \de<\om^{\ga_j}$.

Since $c_\be(n)\in RT_\rho$ we find that the density for $L$ exists for all linear sets.

Since by the same proof as before also in this situation the intersection of linear sets is either empty or again a linear set we have proved.

\begin{theo}
Let $\eo\leq\be=\om^\ga$. Then  limit densities exist for all semi linear subsets of $\be.$
\end{theo}
\begin{proof}
\end{proof}

Let ${\mathbb Q}(\rho)$ be the least field containing $\rho$.
\begin{theo} Let $\eo\leq\be=\om^\ga$. Then the limit densities are elements of ${\mathbb Q}(\rho)$ for all semi linear subsets of $\be.$
\end{theo}
\begin{proof}
The limit densities resulting in our setting from Schur's theorem for linear sets are in ${\mathbb Q}(\rho)$ .
The resulting limiting densities for semi linear sets are formed by taking finite sums with integer coefficients.
This yields the assertion.\end{proof}

Finally let us briefly discuss the case where $\be$ is the Howard Bachmann ordinal.
Every ordinal from the standard notation system for the Howard ordinal (when it is built on the $\vartheta$ function) can be 
we written as a finite multiset over terms of the form $D_0\al$ and $D_1\al$ where $\al$ is again such an ordinal. (See, for example, \cite{Buchholz} for a proof.)
This means that in the Flajolet Sedgewick notation \cite{Flajolet} we find
$OT=\mathrm{Mult}(\{D_0,D_1\}\times OT)$.
Every countable ordinal in $OT$ can be written a multiset over terms of the form $D_0\al$ with $\al\in OT$.
This means that for counting these numbers we need to count the class $CT=\mathrm{Mult}(\{D_0\}\times OT)$.
For the induced generating functions this means that $OT(x)=\exp (\sum_{i=1}^\infty \frac{x^{i\cdot 2}\cdot OT(x^i)}{i})$ and that
$CT(x)=\exp (\sum_{i=1}^\infty \frac{x^{i }\cdot OT(x^i)}{i})$.
These will be tree generating functions with radius of convergence $<1$ so that our analysis applies.

\section{Monadic second order (Cesaro) limit laws for natural well orderings}
Let $\sgn(\al_0)$ be $1$ if $\al_0>0$ and let $\sgn(\al_0)$ be $0$ otherwise.
Let us recall the well known theorem by B\"uchi (theorem 4.8 in \cite{Buchi}) on the countable spectrum 
of monadic second order sentences.

\begin{theo}
Let $\varphi$ be a monadic second order sentence in the language of linear orders.
Then there exists a finite number $r$ and there exist a finite set $K$, an element $a\in K$,
a subset $W\subseteq K$ and operations $F_0,\ldots,F_{r+1}$ on $K$ such that for
all countable ordinals $\al$ of the
form $\al=\om^{r+1}\cdot \al_0+\om^r\cdot k_{r}+\ldots+\om^0\cdot k_0$ we have the
equivalence:

$\al\models \varphi $ iff
$F_0^{k_0}\cdots F_{r}^{k_r}F_{r+1}^{\sgn(\al_0)}(a)\in W$.

Moreover $F_{r+1}^2=F_{r+1}$ and $i<j\leq r$ yields $F_jF_i=F_j$.
\end{theo}

So let us assume a monadic second order sentence $\varphi$  in the language of linear orders
is given and choose $p$, $K,a,W,F_0,\ldots,F_{r+1}$ according to B\"uchi's theorem.
Then clearly for every $i$ we find numbers $a_i$ and $b_i$ such that
$F_i^{a_i}=F_i^{a_i+b_i}$ since $F_i:K\to K$ and $K$ is finite.
For $\al=\om^{r+1}\cdot \al_0+\om^{r}\cdot k_{r}+\ldots+\om^0\cdot k_0$
we then find that $\al\models \varphi$ iff
$\overline{\al}\models \varphi$ where
$\overline{\al}=\om^{r+1}\cdot \al_0+\om^r\cdot \overline{k_{r}}+\ldots+\om^0\cdot \overline{k_0}$.
Here $\overline{k_i}=k_i$ if $k_i<a_i$ and $\overline{k_i}=a_i+\mu c: k_i-a_i =c (mod \;b_i)$.
So the spectrum of $\varphi$ consists of a semi linear set for which we have proved
that limit densities (for $\be\geq \om^\om$) or Cesaro limit densities (for $\om\leq\be<\om^\om$) exist.

This yields for infinite ordinals $\be<\om^\om$  a monadic second order Cesaro limit law and
for $\be\geq \om^\om$
a monadic second order limit law.
Moreover the proofs yield that these limits will always be rational numbers when $\be<\eo$.

\section{Weak Monadic second order limit laws for ordinals in the presence of addition and multiplication}
We can define $\langle \om,+,\cdot\rangle$ in the weak monadic second order language over any infinite structure
$\langle \al,+\rangle$.
To define $\cdot$ we use the well known fact that $x$ divides $y$ is definable on the smallest limit element by the following description:
There is a finite set $X$ such that 
$x\in X$ and such that for every not maximal element $v\in X$ we have $v+x\in X$ and  and such that $y$ is the maximal element of $X$ and $x$ is the minimal element in $X$ and
for every element  $v$ in $X$ which is not $x$ there exists a $w\in X$ such that $w+x=v$.
Then we can define squaring over $\om$ by $y,y+1$ both divide $x+y$ and for all $z<x+y$ if $y$ divides z then
 then it is not the case that $y+1$ divides $u$.
Then multiplication is defined by $x\cdot y=z$ if 
there are $u,v,w(u=x^2\wedge v=y^2\wedge w=(x+y)^2\wedge w=u+v+z+z)$.

This yields that no algorithm can separate $\varphi$ ( from the weak monadic second order language)
with $\delta^\beta_n(\varphi)\to 0$ from $\delta^\beta_n(\varphi)\to 1$
for $\be\geq \om^2$.

By additional results of B\"uchi's about the ordinal spectrum of weak monadic second order logc
we obtain from the previous results of this paper weak monadic second order limit laws for $\om\leq \be\leq \eo$.

Weak monadic second order limit laws lead to  first order limit laws with respect to $L(<,+)$ for classes of structures $\{\om^\al: \al<\be\}$ where
$\om\leq \be\leq \eo$. These can be inherited because by a theorem of Ehrenfeucht for any choice of $\de,\de'$ the structure
$\langle \om^{\om^\om\cdot \de+\al}\rangle$ is $WMSO(<,+)$ elementarily equivalent with $\langle \om^{\om^\om\cdot \de'+\al}\rangle$
since $\langle{\om^\om\cdot \de+\al}\rangle$ is $WMSO(<)$ elementarily equivalent with $\langle {\om^\om\cdot \de'+\al}\rangle$ (for a proof confer, e.g. \cite{Rosenstein}).
Moreover this leads to first order limit laws for $L(<,+,\cdot)$ for classes of structures $\{\om^{\om^\al}:\al<\be\}$ where
$\om\leq \be\leq \eo$. This is again a consequence of another result of Ehrenfeucht since the structure
$\langle{\om^{\om^\om\cdot \de+\al}}\rangle$ is $WMSO(<,+)$ elementarily equivalent with $\langle {\om^{\om^\om\cdot \de'+\al}}\rangle$.

{\bf Final remarks:}  

\begin{enumerate}
\item The corresponding results for the multiplicative setting which are defined relative to the Matula coding (when ordinals not exceeding $\eo$ are concerned) follow from our previous analysis together with
Theorem 9.53 
(the multiplicative version
of Schur's theorem) 
and Theorem 8.30 
(the multiplicative version of Hua's theorem) 
 in \cite{Burris}.
\item We have shown monadic second order limit laws for $\be\geq\eo$  in the additive
setting. We intend to investigate whether similar results holds in the
multiplicative setting.
\item A very exciting extension of our work concerns logical limit laws for
uncountable ordinals. Corresponding density notions can be induced by working with
ordinal notation systems for the Bachmann Howard ordinal and farer reaching notation systems. We believe that 
monadic second order limit laws will hold for all ordinals (including the
uncountable ones) from such a notation system.
For this B\"uchi's theorem for ordinals less than $\om_2$ seems to be applicable
\cite{BuchiZaiontz}. Recent results by Itay Neeman \cite{Neemana,Neemanb} seem to pave the way to study limit laws for ordinals above  $\om_2$ but we quit at this point.
\end{enumerate}

\end{document}